\documentclass[a4paper,reqno]{amsart}
\usepackage{geometry}
\usepackage{mathrsfs}
\usepackage{syntonly}
\usepackage{amsmath}
\usepackage{amsthm}
\usepackage{amsfonts}
\usepackage{amssymb}
\usepackage{latexsym}
\usepackage{amscd,amssymb,amsopn,amsmath,amsthm,graphics,amsfonts,mathrsfs,accents,enumerate,verbatim,calc}
\usepackage[dvips]{graphicx}
\usepackage[colorlinks=true,linkcolor=blue,citecolor=blue]{hyperref}
\input xy
\xyoption{all}
\def\del{\delta}

\usepackage{color}

\numberwithin{equation}{section}
\theoremstyle{plain}

\newtheorem{thm}{Theorem}[section]
\newtheorem{cor}[thm]{Corollary}
\newtheorem{lem}[thm]{Lemma}
\newtheorem{prop}[thm]{Proposition}

\newtheorem{defn}[thm]{Definition}
\newtheorem{exm}[thm]{Example}

\theoremstyle{remark}
\newtheorem{rem}[thm]{Remark}

\newdir{ >}{{}*!/-5pt/\dir{>}}

\renewcommand{\mod}{\operatorname{mod}\nolimits}

\newcommand{\add}{\operatorname{add}\nolimits}
\newcommand{\ann}{\operatorname{ann}\nolimits}
\newcommand{\Fac}{\operatorname{Fac}\nolimits}

\newcommand{\Hom}{\operatorname{Hom}\nolimits}

\newcommand{\End}{\operatorname{End}\nolimits}
\renewcommand{\Im}{\operatorname{Im}\nolimits}

\newcommand{\Ext}{\operatorname{Ext}\nolimits}

\newcommand{\pd}{\operatorname{pd}\nolimits}
\newcommand{\id}{\operatorname{id}\nolimits}

\newcommand{\Cone}{\operatorname{Cone}\nolimits}
\newcommand{\CoCone}{\operatorname{CoCone}\nolimits}

\newcommand{\M}{\mathcal M}

\newcommand{\B}{\mathcal B}
\newcommand{\uB}{\underline{\B}}
\newcommand{\oB}{\overline{\B}}
\newcommand{\U}{\mathcal U}
\newcommand{\V}{\mathcal V}
\newcommand{\A}{\mathcal A}
\newcommand{\W}{\mathcal W}
\newcommand{\h}{\mathcal H}

\newcommand{\T}{\mathcal T}

\newcommand{\I}{\mathcal I}
\newcommand{\D}{\mathcal D}

\newcommand{\K}{\mathcal K}

\newcommand{\R}{\mathcal R}
\newcommand{\X}{\mathcal X}

\newcommand{\C}{\mathcal C}

\newcommand{\EE}{\mathbb E}

\newcommand{\svecv}[2]{\left(\begin{smallmatrix}
      #1 \\
      #2
    \end{smallmatrix}\right)}

\newcommand{\svech}[2]{\left(\begin{smallmatrix}
      #1 & #2
\end{smallmatrix}\right)}

\renewcommand{\emph}{\textit}
\renewcommand{\phi}{\varphi}

\SelectTips{cm}{10}

\begin{document}

\title{On the relation between relative rigid and support tilting}\footnote{Yu Liu was supported by the Fundamental Research Funds for the Central Universities (Grant No.  2682019CX51) and the National Natural Science Foundation of China (Grant No. 11901479). Panyue Zhou was supported by the National Natural Science Foundation of China (Grant Nos. 11901190 and 11671221), and by the Hunan Provincial Natural Science Foundation of China (Grant No. 2018JJ3205), and by the Scientific Research Fund of Hunan Provincial Education Department (Grant No. 19B239).}
\author{Yu Liu and Panyue Zhou}
\address{School of Mathematics, Southwest Jiaotong University, 610031, Chengdu, Sichuan, People's Republic of China}
\email{liuyu86@swjtu.edu.cn}
\address{College of Mathematics, Hunan Institute of Science and Technology, 414006, Yueyang, Hunan, People's Republic of China}
\email{panyuezhou@163.com}
\begin{abstract}
Let $\B$ be an extriangulated category with enough projectives and enough injectives. Let
 $\C$ be a fully rigid subcategory of $\B$ which admits a twin cotorsion pair $((\C,\K),(\K,\D))$.
The quotient category $\B/\K$ is abelian, we assume that it is hereditary and has finite length. In this article, we study the relation between support tilting subcategories of $\B/\K$ and maximal relative rigid subcategories of $\B$. More precisely, we show that the image of any cluster tilting subcategory of $\B$ is support tilting in $\B/\K$ and any support tilting subcategory in $\B/\K$ can be lifted to a unique relative maximal rigid subcategory in $\B$. We also give a bijection between these two classes of subcategories if $\C$ is generated by an object.
\end{abstract}

\maketitle

\section{Introduction}

In mathematics, especially representation theory, classical tilting theory describes a way to relate the module categories of two algebras using so-called tilting modules and associated tilting functors. Classical tilting theory was motivated by the reflection functors introduced by Bernstein, Gelfand and  Ponomarev \cite{BGP}. These functors were reformulated by Auslander, Platzeck and Reiten \cite{APR}, and generalized by Brenner and Butler \cite{BB}.

Inspired by classical tilting theory, Buan, Marsh, Reineke, Reiten and Todorov \cite{BMRRT} introduced
cluster tilting objects in the context of cluster categories associated to a hereditary algebra, in order to categorify certain phenomena occurring in the theory
of cluster algebras developed by Fomin and Zelevinsky \cite{FZ}. Cluster categories have led to new developments in the theory of the canonical basis and its dual. They are providing insight into cluster algebras
and the related combinatorics, and have also been used to establish a new kind
of classical tilting theory, known as cluster tilting theory, which generalises APR-tilting for
hereditary algebras.

In the following result, which is a part of main theorem in \cite{IT}, Ingalls and Thomas introduced the concept of support tilting modules and established a relation between cluster tilting and classical tilting.

\begin{defn}
Let $\Lambda$ be a finite dimensional algebra. A module $T\in \mod \Lambda$ is called support tilting if $\Ext^1_{\Lambda}(T,T)=0$ and it is a tilting module for the algebra $\Lambda/\ann T$, where $\ann T$ is the annihilator of $T$.
\end{defn}

\begin{thm}{\rm (Ingalls and Thomas)}
Let $Q$ be a finite quiver without loops and cycles and  $\C$ be the cluster category of type $Q$
over an algebraically closed field $k$.
Then there is a bijection between the isomorphism classes of basic cluster tilting objects of $\C$ and the isomorphism classes of
basic support tilting modules in ${\rm mod}kQ$.
\end{thm}

Cluster tilting theory gives a way to construct abelian categories from triangulated categories.
Koenig and Zhu \cite[Theorem 3.2]{KZ} showed that if $\C$ is a triangulated category and $\X$ is a cluster tilting  subcategory of $\C$, then
the quotient category $\C/\X$ is an abelian category. Moreover,  the category $\C/\X$ is Gorenstein of Gorenstein dimension at most one \cite[Theorem 4.3]{KZ}, which implies that it is either hereditary or of infinite global dimension.

Now we assume that $k$ is an algebraically closed field and $\C$ is a $k$-linear triangulated category
with finite dimensional Hom spaces and split idempotents which has Serre functor $\mathbb{S}$.
Let $\X$ be a cluster tilting subcategory of $\C$. Suppose that $\C/\X$ is hereditary.
Holm and J{\o}rgensen \cite[Definition 2.1]{HJ} introduced the notion of support tilting subcategories which is a generalization of support tilting modules and proved the following theorem.

\begin{thm}{\rm (Holm and J{\o}rgensen)}
 Let $\W$ be a cluster tilting subcategory of $\C$. Then the image $\overline{\W}$ is a support tilting subcategory of $\C/\X$.
\end{thm}
Under certain conditions, Holm and J{\o}rgensen also showed that  the inverse direction of this theorem is also true.
This gives a category version of Ingalls-Thomas's result.

In \cite{AIR}, Adachi, Iyama and Reiten introduced a generalization of classical tilting theory, which is called $\tau$-tilting theory. According to \cite[Proposition 2.2]{AIR}, if $\Lambda$ is a finite dimensional algebra, then any $\tau$-tilting module is support tilting by definition. They also proved that for a $2$-Calabi-Yau triangulated category $\C$ with a cluster tilting object $T$, there exists a
 bijection between the basic cluster tilting objects in $\C$ and the basic support $\tau$-tilting modules ($\tau$-tilting modules are support $\tau$-tilting) in $\mod\End_\C(T)^{\textrm{op}}$ (which is equivalent to the quotient category $\C/\add T[1]$). This bijection was generalized first by Yang and Zhu \cite{YZ} by introducing the notion of relative cluster tilting objects in a triangulated category with a cluster tilting object, later by Fu, Geng and Liu \cite{FGL}  by introducing the notion of relative rigid objects in a triangulated category $\C$ with a rigid object. When we combine all these results, we can get a hint that support tilting subcategories in the quotient category may have some relations with relative rigid objects in the original category. Hence it is reasonable to investigate what subcategory is a support tilting subcategory in the quotient category related to, under a more general setting.

We want our results to be valid not only on triangulated categories, but also on exact categories. The notion of extriangulated categories was introduced by Nakaoka and Palu in \cite{NP} as a simultaneous generalization of
exact categories and triangulated categories. Exact categories and triangulated categories are extriangulated categories, while there are some other examples of extriangulated categories which are neither exact nor triangulated, see \cite{NP,ZZ1}. In this article, we work under the following settings. Let $k$ be a field and $\B$ be a Krull-Schmidt, Hom-finite, $k$-linear extriangulated category with enough projectives $\mathcal P$ and enough injectives $\mathcal I$ (please see Section 2 for more details of extriangulated category). 

We introduce the notion of fully rigid subcategories, which is a generalization of cluster tilting subcategories (it was defined in a triangulated category by Beligiannis \cite{B}, please see Section 2 for more details).

\begin{defn}
A subcategory $\C$ of $\B$ is called fully rigid if
\begin{itemize}
\item[\rm (1)] it admits a cotorsion pair $(\C,\mathcal K)$ and $\C\neq \mathcal P$;

\item[\rm (2)] any indecomposable object in $\B$ either belongs to $\mathcal K$ or belongs to $$\h:=\{X\in \B\text{ }|\text{ } \text{there is an}~ \EE\text{-triangle }X \longrightarrow C_1 \longrightarrow C_2 \dashrightarrow ~\textrm{where}~ C_1,C_2\in\C \}.$$
 \end{itemize}
\end{defn}

\begin{rem}
It can be shown that when $\C$ is fully rigid, then it is contravariantly finite, rigid and $\B/\C^{\bot_1}$ is an abelian category where $\C^{\bot_1}=\{X\in \B \text{ }|\text{ } \EE(\C,X)=0\}$.
\end{rem}

Under these more general assumptions, we can show the following theorem (please see Theorem \ref{main1} and Corollary \ref{cor1} for more details), which generalizes \cite[Theorem 2.2]{HJ}.

\begin{thm}
Let $\C$ be a fully rigid subcategory of $\B$ and $\M$ be a cluster tilting subcategory of $\B$. We also assume that $\C^{\bot_1}$ is contravariantly finite and $\B/\C^{\bot_1}$ is hereditary. Then $\overline{\M}$ is a support tilting subcategory in $\B/\C^{\bot_1}$.
\end{thm}

The notion of relative rigid can be generalized to an extriangulated category (see \cite{LZII}), and we find the relation between support tilting subcategories in $\B/\C^{\bot_1}$ and maximal relative rigid subcategories (we also call them maximal $\C$-rigid subcategories, please see Section 4 for more details) in $\h$.

\begin{thm}
Let $\C$ be a fully rigid subcategory of $\B$. We also assume that $\C^{\bot_1}$ is contravariantly finite, $\B/\C^{\bot_1}$ is hereditary and has finite length.
\begin{itemize}
\item[(a)] For any support tilting subcategories $\overline{\W}$ in $\B/\C^{\bot_1}$, there is a unique maximal $\C$-rigid subcategory $\X$ in $\h$ such that $\overline{\X}=\overline{\W}$ in $\B/\C^{\bot_1}$.
\vspace{1mm}
\item[(b)] Moreover, if $\C=\add C$, there is a bijection between support tilting subcategories $\overline{\W}=\add W$ in $\B/\C^{\bot_1}$ and maximal $\C$-rigid subcategory $\X=\add X$ in $\h$.
\end{itemize}
\end{thm}


This article is organized as follows. In Section 2, we review some elementary properties of extriangulated category
that we need and show some results which will be used later. In Section 3, we show our first main result (see Theorem \ref{main1} and Corollary \ref{cor1}). In Section 4, we study the relation between support tilting subcategories in the quotient category and maximal relative rigid subcategories in $\B$. In Section 5, we give an example to explain our main results.

\medskip
\section{Preliminaries}

\subsection{Extriangulated categories}
Let us briefly recall the definition and some basic properties of extriangulated category from \cite{NP}.
We omit some details here, but the reader can find
them in \cite{NP}.

Let $\B$ be an additive category equipped with an additive bifunctor
$$\mathbb{E}: \B^{\rm op}\times \B\rightarrow {\rm Ab},$$
where ${\rm Ab}$ is the category of abelian groups. For any objects $A, C\in\B$, an element $\delta\in \mathbb{E}(C,A)$ is called an $\mathbb{E}$-extension.
Let $\mathfrak{s}$ be a correspondence which associates an equivalence class $$\mathfrak{s}(\delta)=\xymatrix@C=0.8cm{[A\ar[r]^x
 &B\ar[r]^y&C]}$$ to any $\mathbb{E}$-extension $\delta\in\mathbb{E}(C, A)$. This $\mathfrak{s}$ is called a {\it realization} of $\mathbb{E}$, if it makes the diagrams in \cite[Definition 2.9]{NP} commutative.
 A triplet $(\B, \mathbb{E}, \mathfrak{s})$ is called an {\it extriangulated category} if it satisfies the following conditions.
\begin{enumerate}
\setlength{\itemsep}{2.5pt}
\item $\mathbb{E}\colon\B^{\rm op}\times \B\rightarrow \rm{Ab}$ is an additive bifunctor.

\item $\mathfrak{s}$ is an additive realization of $\mathbb{E}$.

\item $\mathbb{E}$ and $\mathfrak{s}$  satisfy the compatibility conditions in \cite[Definition 2.12]{NP}.
 \end{enumerate}
\smallskip

We will use the following terminology.

\begin{defn}\label{dein}{\cite{NP}}
Let $(\B,\EE,\mathfrak{s})$ be an extriangulated category.
\begin{itemize}
\setlength{\itemsep}{2.5pt}
\item[{\rm (1)}] A sequence $A\xrightarrow{~x~}B\xrightarrow{~y~}C$ is called a {\it conflation} if it realizes some $\EE$-extension $\del\in\EE(C,A)$. In this case, $x$ is called an {\it inflation} and $y$ is called a {\it deflation}.

\item[{\rm (2)}] If a conflation $A\xrightarrow{~x~}B\xrightarrow{~y~}C$ realizes $\delta\in\mathbb{E}(C,A)$, we call the pair $( A\xrightarrow{~x~}B\xrightarrow{~y~}C,\delta)$ an {\it $\EE$-triangle}, and write it in the following way.
$$A\overset{x}{\longrightarrow}B\overset{y}{\longrightarrow}C\overset{\delta}{\dashrightarrow}$$
We usually do not write this $``\delta"$ if it is not used in the argument.
\item[{\rm (3)}] Let $A\overset{x}{\longrightarrow}B\overset{y}{\longrightarrow}C\overset{\delta}{\dashrightarrow}$ and $A^{\prime}\overset{x^{\prime}}{\longrightarrow}B^{\prime}\overset{y^{\prime}}{\longrightarrow}C^{\prime}\overset{\delta^{\prime}}{\dashrightarrow}$ be any pair of $\EE$-triangles. If a triplet $(a,b,c)$ realizes $(a,c)\colon\delta\to\delta^{\prime}$, then we write it as
$$\xymatrix{
A \ar[r]^x \ar[d]^a & B\ar[r]^y \ar[d]^{b} & C\ar@{-->}[r]^{\del}\ar[d]^c&\\
A'\ar[r]^{x'} & B' \ar[r]^{y'} & C'\ar@{-->}[r]^{\del'} &}$$
and call $(a,b,c)$ a {\it morphism of $\EE$-triangles}.

\item[{\rm (4)}] An object $P\in\B$ is called {\it projective} if
for any $\EE$-triangle $A\overset{x}{\longrightarrow}B\overset{y}{\longrightarrow}C\overset{\delta}{\dashrightarrow}$ and any morphism $c\in\B(P,C)$, there exists $b\in\B(P,B)$ satisfying $yb=c$.
We denote the subcategory of projective objects by $\mathcal P\subseteq\B$. Dually, the subcategory of injective objects is denoted by $\I\subseteq\B$.

\item[{\rm (5)}] We say that $\B$ {\it has enough projective objects} if
for any object $C\in\B$, there exists an $\EE$-triangle
$A\overset{x}{\longrightarrow}P\overset{y}{\longrightarrow}C\overset{\delta}{\dashrightarrow}$
satisfying $P\in\mathcal P$. Dually we can define $\B$ {\it has enough injective objects}.

\end{itemize}
\end{defn}

Let $k$ be a field and $(\B,\mathbb{E},\mathfrak{s})$ be a Krull-Schmidt, Hom-finite, $k$-linear extriangulated category with enough projectives $\mathcal P$ and enough injectives $\mathcal I$.
\smallskip

By \cite{NP}, we give the following useful remark, which will be used later in the proofs.

\begin{rem}\label{useful}
Let $\xymatrix{A\ar[r]^a &B \ar[r]^b &C \ar@{-->}[r] &}$ and $\xymatrix{X\ar[r]^x &Y \ar[r]^y &Z \ar@{-->}[r] &}$ be two $\EE$-triangles. Then
\begin{itemize}
\item[(a)] In the following commutative diagram
$$\xymatrix{
X\ar[r]^x \ar[d]_f &Y \ar[d]^g \ar[r]^y &Z \ar[d]^h \ar@{-->}[r] &\\
A\ar[r]^a &B \ar[r]^b &C \ar@{-->}[r] &}
$$
$f$ factors through $x$ if and only if $h$ factors through $b$.
\item[(b)] In the following commutative diagram
$$\xymatrix{
A\ar[r]^a \ar[d]_s &B \ar[d]^r \ar[r]^b &C \ar[d]^t \ar@{-->}[r] &\\
X\ar[r]^x \ar[d]_f &Y \ar[d]^g \ar[r]^y &Z \ar[d]^h \ar@{-->}[r] &\\
A\ar[r]^a &B \ar[r]^b &C \ar@{-->}[r] &}
$$
$fs=1_A$ implies $B$ is a direct summand of $C\oplus Y$ and $C$ is a direct summand of $Z\oplus B$; $ht=1_C$ implies $B$ is a direct summand of $A\oplus Y$ and $A$ is a direct summand of $X\oplus B$.
\item[(c)] Let $A\overset{x}{\longrightarrow}B\overset{y}{\longrightarrow}C\overset{\delta}{\dashrightarrow}$ be any $\EE$-triangle, let $f\colon A\rightarrow D$ be any morphism, and let $D\overset{d}{\longrightarrow}E\overset{e}{\longrightarrow}C\overset{f_{\ast}\delta}{\dashrightarrow}$ be any $\EE$-triangle realizing $f_{\ast}\delta$. Then there is a morphism $g$ which gives a morphism of $\EE$-triangles
$$\xymatrix{
A \ar[r]^{x} \ar[d]_f &B \ar[r]^{y} \ar[d]^g &C \ar@{=}[d]\ar@{-->}[r]^{\delta}&\\
D \ar[r]_{d} &E \ar[r]_{e} &C\ar@{-->}[r]_{f_{\ast}\delta}&
}
$$
and moreover, the sequence $A\overset{\svecv{f}{x}}{\longrightarrow}D\oplus B\overset{\svech{d}{-g}}{\longrightarrow}E\overset{e^{\ast}\delta}{\dashrightarrow}$ becomes an $\EE$-triangle.
\end{itemize}
\end{rem}

\begin{defn}
Let $\B'$ and $\B''$ be two subcategories of $\B$.
\begin{itemize}
\item[(a)] Denote by $\CoCone(\B',\B'')$ the subcategory
$$\{X\in \B \text{ }|\text{ } ~\textrm{there exists an}~ \text{ } \EE\text{-triangle } \xymatrix@C=0.8cm@R0.6cm{ X \ar[r] &B' \ar[r] &B'' \ar@{-->}[r] &} \text{, }B'\in \B' \text{ and }B''\in \B'' \};$$
\item[(b)] Denote by $\Cone(\B',\B'')$ the subcategory
$$\{X\in \B \text{ }|\text{ } ~\textrm{there exists an}~ \text{ } \EE\text{-triangle } \xymatrix@C=0.8cm@R0.6cm{B' \ar[r] &B'' \ar[r] &X \ar@{-->}[r] &} \text{, }B'\in \B' \text{ and }B''\in \B''  \};$$
\item[(c)] Let $\Omega \B'=\CoCone(\mathcal P,\B')$. We write an object $D$ in the form $\Omega B$ if it admits an $\EE$-triangle $\xymatrix@C=0.8cm@R0.6cm{D \ar[r] &P \ar[r] &B \ar@{-->}[r] &}$ where $P\in \mathcal P$;
\item[(d)] Let $\Sigma \B'=\Cone(\B',\mathcal I)$. We write an object $D'$ in the form $\Sigma B'$ if it admits an $\EE$-triangle $\xymatrix@C=0.8cm@R0.6cm{B' \ar[r] &I \ar[r] &D' \ar@{-->}[r] & }$ where $I\in \mathcal I$.
\end{itemize}

\end{defn}

\begin{defn}\cite[Definition 2.10]{ZZ2}
 Let $\C$ be a subcategory of $\B$.
\begin{itemize}
\item[\rm (1)] $\C$ is called rigid if $\EE(\C,\C)=0$;

\item[\rm (2)] $\C$ is called \emph{weak cluster tilting}, if it satisfies the
following conditions:
\begin{itemize}

\item[\rm (a)] $X\in\C$ if and only if $\EE (X, \C) = 0$;

\item[\rm (b)] $X\in\C$ if and only if $\EE(\C, X) = 0$.
\end{itemize}
\item[\rm (3)] $\C$ is called \emph{cluster tilting} if it is functorially finite and weak cluster tilting.
\end{itemize}
\end{defn}

From this definition, we know that if $\C$ is (weak) cluster tilting, then it is closed under isomorphisms and direct summands, and contains all the projectives and all the injectives.

\begin{defn}\cite[Definition 2.1 and Definition 4.12]{NP}
Let $\U$ and $\V$ be two subcategories of $\B$ which are closed under direct summands. We call $(\U,\V)$ a \emph{cotorsion pair} if it satisfies the following conditions:
\begin{itemize}
\setlength{\itemsep}{2.5pt}
\item[{\rm (a)}] $\EE(\U,\V)=0$.

\item[{\rm (b)}] For any object $B\in \B$, there exist two $\EE$-triangles
\begin{align*}
V_B\rightarrow U_B\rightarrow B{\dashrightarrow},\quad
B\rightarrow V^B\rightarrow U^B{\dashrightarrow}
\end{align*}
satisfying $U_B,U^B\in \U$ and $V_B,V^B\in \V$.
\end{itemize}
 Let $(\mathcal S,\T )$ and $(\U,\V)$ be cotorsion pairs on $\B$. Then
the pair $((\mathcal S,\T ),(\U,\V))$ is called a \emph{twin cotorsion pair} if it satisfies
$\EE(\mathcal S,\V) = 0$.
\end{defn}

\begin{defn}
A rigid subcategory $\C$ of $\B$ is called \emph{fully rigid} if
\begin{itemize}
\item[\rm (1)] it admits a cotorsion pair $(\C,\K)$ and $\C\neq\mathcal P$;

\item[\rm (2)]  any indecomposable object in $\B$ either belongs to $\mathcal K$ or belongs to $\h:=\CoCone(\C,\C)$.
 \end{itemize}
\end{defn}

From this definition, we know that cluster tilting subcategories are fully rigid. But fully rigid subcategories are not necessarily cluster tilting, see Example \ref{ex1} in Section 5.

\smallskip

In this article, we always assume $\C$ is a fully rigid subcategory of $\B$ which admits a twin cotorsion pair $((\C,\K),(\K,\D))$.

\begin{lem}\label{summand}
$\CoCone(\C,\C)$ and $\Cone(\D,\D)$ are closed under direct summands.
\end{lem}

\begin{proof}
We only show that $\CoCone(\C,\C)$ is closed under direct summands, the other half is by dual.\\
Let $X_1\oplus X_2\in \CoCone(\C,\C)$. It admits an $\EE$-triangle $\xymatrix{X_1\oplus X_2 \ar[r]^-{\svech{q_1}{q_2}} &C^1 \ar[r]^{p} &C^2 \ar@{-->}[r] &}$ where $C^1,C^2\in \C$. Since $q_1=(q_1~q_2)\binom{1}{0}$, we have that $q_1$ is an inflation. So $q_1$ admits an $\EE$-triangle $\xymatrix{X_1 \ar[r]^{q_1} &C^1 \ar[r]^{p_1} &C \ar@{-->}[r] &}$. Thus we get the following commutative diagram:
$$\xymatrix{
X_1 \ar[d]_-{\svecv{1}{0}} \ar[r]^{q_1} &C^1 \ar@{=}[d] \ar[r]^{p_1} &C \ar[d] \ar@{-->}[r] & \\
X_1\oplus X_2 \ar[r]^-{\svech{q_1}{q_2}} \ar[d]_-{\svech{1}{0}} &C^1 \ar[r]^p \ar[d]^a &C^2 \ar[d] \ar@{-->}[r] &\\
X_1 \ar[r]^{q_1} &C^1 \ar[r]^{p_1} &C \ar@{-->}[r] &.
}
$$
Then $C$ is a direct summand of $C^1\oplus C^2\in \C$. Since $\C$ is closed under direct summands, we have $X_1\in \CoCone(\C,\C)$. Hence $\CoCone(\C,\C)$ is closed under direct summands.
\end{proof}

For objects $A,B\in\B$ and a subcategory $\B'$ of $\B$, let $[\B'](A,B)$ be the subgroup of $\Hom_{\B}(A,B)$ consisting of morphisms which factor through objects in $\B'$. For a morphism $x:A\to X$ (or $x:X\to B$), let $[x](A,B)$ be the subgroup of $\Hom_{\B}(A,B)$ consisting of morphisms which factor through $x$, let $[\B',x](A,B)$ be the subgroup of $\Hom_{\B}(A,B)$ consisting of morphisms which factor through $\B'$ and the morphism $x$.  For another morphism $x':A\to X$ (or $x':X\to B$),  let $[x,x'](A,B)$ be the subgroup of $\Hom_{\B}(A,B)$ consisting of morphisms which factor through morphism $x$ and $x'$.
\medskip

We denote $\B'/\mathcal P$ by $\underline \B'$ if $\mathcal P\subseteq \B'\subseteq \B$. For any morphism $f\colon A\to B$ in $\B$, we denote by $\underline{f}$ the image of $f$ under the natural quotient functor $\B\to \uB$. We denote $\B'/\K$ by $\overline {\B'}$ if $\K\subseteq \B'\subseteq \B$. For any morphism $g\colon A\to B$ in $\B$, we denote by $\overline{g}$ the image of $g$ under the natural quotient functor $\B\to \oB$.

\begin{defn}\cite[Definition 3.3]{LN}
Let $(\B,\EE,\mathfrak{s})$ be an extriangulated category and $\A$ be an
abelian category. An additive functor $H\colon \B\to \A$ is called cohomological, if any
$\EE$-triangle $$\xymatrix{A\ar[r]^f &B \ar[r]^g &C \ar@{-->}[r]&}$$
yields an exact sequence
$$H(A)\xrightarrow{H(f)} H(B)\xrightarrow{H(g)} H(C)$$ in $\A$.
\end{defn}

The following property was shown in \cite{LN}, which is very useful in this article.

\begin{prop}
The natural quotient functor $\pi:\B\to \oB$ is a cohomological functor.
\end{prop}

Let $\W$ and $\R$ be subcategories of $\B$. We denote
$$\W_\R=\{W\in \W \text{ }|\text{ }W \text{ has no non-zero direct summands in }\R \}.$$

\begin{lem}\label{heart}
If $X\in \h_\C$, then $X\in \Cone(\D,\D)$.
\end{lem}

\begin{proof}
Let $X\in \h_\C$ be an indecomposable object.  It admits the following commutative diagram
$$\xymatrix{
D_X \ar[r] \ar@{=}[d] &K_X \ar[r] \ar[d] &X \ar[d]^x \ar@{-->}[r] &\\
D_X \ar[r] &D \ar[r] \ar[d] &Y \ar[d] \ar@{-->}[r] &\\
&K \ar@{-->}[d] \ar@{=}[r] &K \ar@{-->}[d]\\
& &
}
$$
where $D_X,D\in \D$ and $K_X,K\in \K$. Then $Y\in \Cone(\D,\D)$, by applying $\pi$ to this diagram, we get an isomorphism $X\xrightarrow{\overline x} Y$. Then $x$ is a section and $X$ is a direct summand of $Y$, by Lemma \ref{summand}, $X\in \Cone(\D,\D)$.
\end{proof}

\begin{prop}
We have the following properties:
\begin{itemize}
\item[(a)] $\oB\simeq \h/\C\simeq \mod\underline {\Omega \C}\simeq \mod\underline \C$;
\item[(b)] the subcategory $\underline {\Omega \C}$ of $\oB$ is the enough projectives;
\item[(c)] the subcategory $(\Sigma \D)/\mathcal I$ of $\oB$ is the enough injectives;
\item[(d)] $\B_\K=\h_\C$.
\end{itemize}
\end{prop}

\begin{proof}
(a) By the definition of a fully rigid subcategory, $\oB$ and $\h/\C$ actually have the same objects. If $A,B\in \h$, then $[\C](A,B)=[\K](A,B)$. Hence $\oB\simeq \h/\C$.

The equivalence $\h/\C\simeq \mod\underline {\Omega \C}$ is proved by \cite[Theorem 1.2]{LZ}.

The equivalence $\h/\C\simeq \mod\underline {\C}$ is proved by \cite[Theorem 3.4]{ZZ2}.

Hence we have $\oB\simeq \h/\C\simeq \mod\underline {\Omega \C}\simeq \mod\underline \C$.
\smallskip

(b) Note that a morphism $f$ in $\Omega \C$ factors through $\K$ if and only if it factors through $\mathcal P$. Since $\oB\simeq \h/\C$ and $\underline {\Omega \C}$ is the enough projectives in $\h/\C$ by \cite[Theorem 4.10]{LN}, we get that the subcategory $\underline {\Omega \C}$ of $\oB$ is the enough projectives.
\smallskip

(c) This is followed by Lemma \ref{heart} and the dual of \cite[Theorem 4.10]{LN}.
\smallskip

(d) Since $\h\cap \K=\C$, by the definition of a fully rigid subcategory, we have $\B_\K=\h_\C$.
\end{proof}

\subsection{Relative rigid subcategories}

We denote by $\pd_{\B}(X)$ the projective dimension of an object $X$ in $\B$ and by $\pd_{\oB}(X)$ (resp. $\id_{\oB}(X)$) the projective (resp. injective) dimension of $X$ in $\oB$.

For an object $X\in \h_\C$, we can always get the following commutative diagram
$$\xymatrix{
\Omega X \ar[r]^{a} \ar[d]^{a'} &\Omega C^1 \ar@{=}[r] \ar[d]^b &\Omega C^1\ar[d]^{p_1}\\
P_X \ar[r]^{b'} \ar[d] &\Omega C^2 \ar[r]^{p_2} \ar[d]^c &P\ar[r]^{q_2} \ar[d]^{q_1} &C^2 \ar@{=}[d] \ar@{-->}[r] & \ar@{}[d]^-{(\maltese)}\\
X \ar@{=}[r] \ar@{-->}[d] &X \ar[r]^x \ar@{-->}[d] &C^1 \ar@{-->}[d] \ar[r]^d &C^2 \ar@{-->}[r] &\\
& & &
}$$
where $C^1,C^2\in \C$, $P,P_X\in \mathcal P$. Then we have the following proposition.

\begin{prop}\label{imp2}
Let $X\in \h_\C$. Then
$\pd_{\oB}(X)\leq 1$ if and only if in the diagram $(\maltese)$, the morphism $a:\Omega X\to \Omega C^1$ factors through $\K$. 
\end{prop}

\begin{proof}

The diagram $(\maltese)$ induces an exact sequence $\Omega X \xrightarrow{\overline a} \Omega C^1 \xrightarrow{\overline b} \Omega C^2 \xrightarrow{\overline c} X\to 0$ in $\oB$. 
If $a$ factors through $\K$, $X$ admits a short exact sequence $0\to \Omega C^1 \to \Omega C^2 \to X \to 0$ in $\oB$,
hence $\pd_{\oB}(X)\leq 1$.
\smallskip

Now we prove the ``only if" part, The proof is divided into two steps.

\smallskip

(1) We show that $\underline b$ is right minimal in $({\maltese})$.

We claim that $d$ is right minimal, otherwise $d$ can be written as $C^1_1\oplus C^1_2\xrightarrow{\svech{d_1}{0}} C^2$, this implies that $C^1_2$ is a direct summand of $X$. But $X\in \h_\C$, a contradiction. In fact $\underline d$ is also right minimal. Since if we have a morphism $C^1\xrightarrow{c^1} C^1$ such that $\underline {dc^1}=\underline d$, then we have $d(c^1-1_{C^1}):C^1 \xrightarrow{p^1} P'\xrightarrow{p^2} C^2$ where $P'\in \mathcal P$. The morphism $p^2$ factors through $d$, we have $p^2:P'\xrightarrow{p^3} C^1\xrightarrow{d} C^2$. Hence $d(c^1-1_{C^1})=dp^3p^1$, then $d(c^1-p^3p^1)=d$, which implies $c^1-p^3p^1$ is an isomorphism, thus $\underline {c^1}$ is an isomorphism. It follows that $\underline b$ is right minimal. Let $b_1:\Omega C^1\to \Omega C^1$ be a morphism such that $\underline {bb_1}=\underline b$. Then we have the following commutative diagram:
$$\xymatrix{
\Omega C^1 \ar[r]^{p_1} \ar[d]^{b_1} &P \ar[r]^{q_1} \ar[d] &C^1 \ar[d]^{c_1} \ar@{-->}[r] &\\
\Omega C^1 \ar[r]^{p_1} \ar[d]^{b} &P \ar@{->>}[r]^{q_1} \ar@{=}[d] &C^1 \ar[d]^d \ar@{-->}[r] &\\
\Omega C^2 \ar[r]^{p_2}  &P \ar[r]^{q_2} &C^2 \ar@{-->}[r] &.
}
$$
Since $\underline {bb_1}=\underline b$, we obtain that $bb_1-b$ factors through $p_1$, this implies $dc_1-d$ factors though $q_2$. Hence $\underline {dc_1}=\underline d$, which means $\underline {c_1}$ is an isomorphism. Let $\underline {c_2}$ be the inverse of $\underline {c_1}$. Then we have the following commutative diagram:
$$\xymatrix{
\Omega C^1 \ar[r]^{p_1} \ar[d]^{b_1} &P \ar[r]^{q_1} \ar[d] &C^1 \ar[d]^{c_1} \ar@{-->}[r] &\\
\Omega C^1 \ar[r]^{p_1} \ar[d]^{b_2} &P \ar[r]^{q_1} \ar[d] &C^1 \ar[d]^{c_2} \ar@{-->}[r] &\\
\Omega C^1 \ar[r]^{p_1}  &P \ar[r]^{q_1} &C^1 \ar@{-->}[r] &.
}
$$
Since $1_{C^1}-c_2c_1$ factors through $\mathcal P$, it factors through $q_1$. Thus $1_{\Omega C^1}-b_2b_1$ factors through $p^1$, we have $\underline {b_2b_1}=\underline 1_{\Omega C^1}$. By the similar argument we can find another morphism $b_2':\Omega C^1\to \Omega C^1$ such that $\underline {b_1b_2'}=\underline 1_{\Omega C^1}$. Hence $\underline {b_1}$ is an isomorphism and then $\underline b$ is right minimal.
\medskip

(2) We show $\overline a=0$.
\smallskip

Since $X\in \h_\C$, we have $\overline c\neq 0$. If $\overline b=0$, then $\underline b=0$, but it is right minimal, we obtain $\Omega C^1\in \mathcal P$ and $\overline a=0$. Now let $\overline b\neq 0$. Assume $\overline a\neq 0$, then we have the following exact sequence:
$$\xymatrix@C=0.5cm@R0.4cm{\Omega X \ar[rr]^{\overline a} \ar@{->>}[dr]_{\overline {r_4}} &&\Omega C^1 \ar[rr]^{\overline b} \ar@{->>}[dr]_{\overline {r_2}} &&\Omega C^2 \ar[r]^-{\overline c} &X \ar[r] &0\\
&R_2\ar@{ >->}[ur]_{\overline {r_3}}  &&R_1 \ar@{ >->}[ur]_{\overline {r_1}}}
$$
where $\Omega C^1 \xrightarrow{\overline {r_2}} R_1\xrightarrow{\overline {r_1}} \Omega C^2$ is an epic-monic factorization of $\overline b$ and $\Omega X \xrightarrow{\overline {r_4}} R_2\xrightarrow{\overline {r_3}} \Omega C^1$ is an epic-monic factorization of $\overline a$. Since $\pd_{\oB}(X)\leq 1$, we have $R_1\in \overline {\Omega \C}$, hence we get a split short exact sequence $0\to R_2\xrightarrow{\overline {r_3}} \Omega C^1 \xrightarrow{\overline {r_2}} R_1\to 0$ which implies $R_2\in \overline{\Omega \C}$. Thus $R_2$ is a direct summand of $\Omega C^1$. Then the $\EE$-triangle $\xymatrix{\Omega X \ar[r]^-{\svecv{a'}{a}} &P_X\oplus \Omega C^1 \ar[r]^-{\svech{-b'}{b}} &\Omega C^2 \ar@{-->}[r] &}$ has the following form
$$S\oplus R_2\xrightarrow{\left(\begin{smallmatrix}
a_{11}&a_{12}\\
a_{21}&a_{22}
\end{smallmatrix}\right)}T\oplus R_2 \xrightarrow{\left(\begin{smallmatrix}
t&r
\end{smallmatrix}\right)} \Omega C^2\dashrightarrow$$
where $a_{22}$ is an isomorphism. Note that $\svech{\underline t}{\underline r}=\underline b$ is right minimal. We have an isomorphism $$T\oplus R_2 \xrightarrow{\left(\begin{smallmatrix}
1_T&a_{12}\\
0&a_{22}
\end{smallmatrix}\right)} T\oplus R_2$$ such that $\svech{t}{r}\left(\begin{smallmatrix}
1_T&a_{12}\\
0&a_{22}
\end{smallmatrix}\right)=\svech{t}{0}.$ But $\underline b$ is right minimal, this implies $R_2\in \mathcal P$, then $\overline a=0$, a contradiction. Hence $\overline a=0$.
\end{proof}

\begin{prop}\label{imp}
Let $X,Y$ be in $\B$.
\begin{itemize}
\item[(1)] If $X\in \h_{\C}$ and $\pd_{\oB}(X)\leq 1$, then $\Ext^1_{\oB}( X, Y)\simeq [\C](X,\Sigma Y)/[\C,i](X,\Sigma Y)$, where morphism $i$ admits the following $\EE$-triangle:
$$\xymatrix{
Y \ar[r] &I \ar[r]^i &\Sigma Y \ar@{-->}[r] &
}
$$
where $I\in \mathcal I$. For convenience, we denote $[\C](X,\Sigma Y)/[\C,i](X,\Sigma Y)$ by $\overline {[\C]}(X,\Sigma Y)$. Moreover, if $\EE(X,Y)=0$, then $\Ext^1_{\oB}(X,Y)=0$.
\item[(2)] If $Y\in \h_{\C}$ and $\id_{\oB}(Y)\leq 1$, then $\Ext^1_{\oB}(X, Y)\simeq [\D](\Omega X, Y)/[\D,p](\Omega X, Y)$, where $p$ admits the following $\EE$-triangle:
$$\xymatrix{
\Omega X \ar[r]^p  & P \ar[r] &X \ar@{-->}[r] &
}
$$
where $P\in \mathcal P$. For convenience, we denote $[\D](\Omega X, Y)/[\D,p](\Omega X, Y)$ by $\underline {[\D]}(\Omega X,Y)$. Moreover, if $\EE(X,Y)=0$, then $\Ext^1_{\oB}(X, Y)=0$.
\end{itemize}
\end{prop}

\begin{proof}
We show (1), (2) is by dual.

If $\pd_{\oB}(X)\leq 1$, by the same method as in Proposition \ref{imp2}, we can get a short exact sequence $0\to \Omega C^1\xrightarrow{\overline b} \Omega C^2\xrightarrow{\overline c} X\to 0$ in $\oB$. Then
$$\Ext^1_{\oB}( X, Y)\simeq \Hom_{\oB}(\Omega C^1, Y)/\Im \Hom_{\oB}(\overline b, Y)\simeq\Hom_{\B}(\Omega C^1,Y)/\Im \Hom_{\B}(b, Y).$$
 To show the second equivalence, you only need to note that $\Hom_{\B}(b, K)$ is full if $K\in \K$. Then any morphism $\overline \alpha$ lies in $\Im \Hom_{\oB}(\overline b, Y)$ if and only if $\alpha$ lies in $\Im \Hom_{\B}(b, Y)$. Let $y\in \Hom_{\B}(\Omega C^1,Y)$. Then we have the following commutative diagram
$$\xymatrix{
\Omega X \ar[r]^q \ar[d]_{a} &P_X\ar[r] \ar[d] &X \ar@{=}[d] \ar@{-->}[r] &\\
\Omega C^1 \ar[r] \ar@{=}[d] &\Omega C^2 \ar[r] \ar[d] &X \ar[d]^x \ar@{-->}[r] &\\
\Omega C^1 \ar[r] \ar[d]_y &P \ar[r] \ar[d] &C^1 \ar@{-->}[r] \ar[d] &\\
Y \ar[r] &I \ar[r]^i &\Sigma Y \ar@{-->}[r] &
}
$$
It follows that
$$\begin{aligned}%
\Hom_{\B}(\Omega C^1,Y)/\Im \Hom_{\B}(b, Y)&\simeq[a](\Omega X ,Y)/[a,q](\Omega X,Y)\\
                     &\simeq [x](X,\Sigma Y)/[x,i](X,\Sigma Y)\\
                     &\simeq \overline{[\C]}(X,\Sigma Y).
\end{aligned}$$
\end{proof}

We introduce some important notions here, which are related to this proposition and will be used in Section 4.

\begin{defn}\label{d1}
An object $X\in \h$ is called $\C$-rigid if $\overline {[\C]}(X,\Sigma X)={[\C]}(X,\Sigma X)$. A subcategory $\X\subseteq \h$ is called $\C$-rigid if $\forall X\in \X$, $X$ is $\C$-rigid.
\end{defn}

\begin{rem}
Assume that $\oB$ is hereditary. According to the preceding proposition, a subcategory $\X$ is $\C$-rigid if and only if $\Ext^1_{\oB}(\overline \X,\overline \X)=0$.
\end{rem}

\begin{defn}
A subcategory $\X\subseteq \h$ is called maximal $\C$-rigid if it satisfies the following conditions:
\begin{itemize}
\item[(a)] $\X$ is $\C$-rigid;
\item[(b)] $\mathcal P\subseteq \X$;
\item[(c)] If $\add(\X\cap \X')\subseteq \h$ is $\C$-rigid, then $\X'\subseteq \X$.
\end{itemize}
An object $X\in \h$ is called maximal $\C$-rigid if $\add (\add X\cup \mathcal P)$ is maximal $\C$-rigid.
\end{defn}



\medskip
\section{Main result I}

Support tilting subcategories were introduced by Holm and J{\o}rgensen, which can be regard as a generalization of support tilting modules.
\begin{defn}\cite[Definition 2.1]{HJ}
To say that $\M$ is a support tilting subcategory of an abelian category $\A$ means that $\M$ is a full subcategory which
\begin{itemize}
\setlength{\itemsep}{2.5pt}
\item is closed under direct sums and direct summands;

\item is funtorially finite in $\A$;

\item satisfies $\Ext^2_{\A}(\M,-)=0$;

\item satisfies $\Ext^1_{\A}(\M,\M)=0$;

\item satisfies that if $A$ is a subquotient of an object from $\M$ such that
$\Ext^1_{\A}(\M,A)=0$, then $B$ is a quotient of an object from $\M$.
\end{itemize}
An object $M$ is called a support tilting object if $\add M$ is support tilting.
\end{defn}

\begin{rem}
According to \cite{AIR}, if $\A=\mod \Lambda$ where $\Lambda$ is a finite dimensional $k$-algebra, any $\tau_{\mathcal A}$-tilting module is support tilting.
\end{rem}


The following corollary is needed in the proof of our first main theorem.

\begin{cor}\label{ses}
Let $X\in \h_\C$ such that $\pd_{\oB}(X)\leq 1$ and $Y\notin \K$. If we have an $\EE$-triangle $$\xymatrix{Y \ar[r]^g &Z \ar[r]^f &X \ar@{-->}[r] &}$$
such that $\overline f$ is an epimorphism in $\oB$, then this $\EE$-triangle induces a short exact sequence $$0\to Y\xrightarrow{~\overline g~}  Z\xrightarrow{~\overline f~} X\to 0.$$
\end{cor}

\begin{proof}
By hypothesis, we already have an exact sequence $ Y\xrightarrow{\overline g} Z\xrightarrow{\overline f}  X\to 0$. It is enough to show $\overline g$ is a monomorphism.

$X$ admits the following commutative diagram
$$\xymatrix{
\Omega C^1 \ar@{=}[r] \ar[d]^b &\Omega C^1\ar[d]^{p_1}\\
\Omega C^2 \ar[r]^{p_2} \ar[d]^c &P\ar[r]^{q_2} \ar[d]^{q_1} &C^2 \ar@{=}[d] \ar@{-->}[r] &\\
X \ar[r]^x \ar@{-->}[d] &C^1 \ar@{-->}[d] \ar[r]^d &C^2 \ar@{-->}[r] &\\
& &
}$$

Since $\overline f$ is an epimorphism and $\Omega C^2$ is projective in $\oB$, there is a morphism $z:\Omega C^2\to Z$ such that $\overline {fz}=\overline c$. Then $fz-c$ factors through $p_2$, that is to say, there is a morphism $q':P\to X$ such that $fz-c=q'p_2$. But $P$ is projective, hence there is a morphism $q'':P\to Z$ such that $fq''=q'$. It follows that $f(z-q''p_2)=c$ and  we have the following commutative diagram
$$
\xymatrix{
\Omega X \ar[r]^q \ar[d]_{a}  &P_X \ar[r] \ar[d] &X \ar@{=}[d] \ar@{-->}[r] &\\
\Omega C^1 \ar[r] \ar[d]_d &\Omega C^2 \ar[r]^c \ar[d]^{z-q''p_2} &X \ar@{=}[d] \ar@{-->}[r] &\\
Y \ar[r]_g &Z \ar[r]_f &X \ar@{-->}[r] &
}
$$
where $a$ factors through $\K$. Then we have an $\EE$-triangle $\xymatrix{\Omega X \ar[r]^-{\svecv{da}{q}} &Y\oplus P_X \ar[r]^-{\svech{g}{*}} &X \ar@{-->}[r] &}$ which induces an exact sequence $\Omega X\xrightarrow{0} Y\xrightarrow{~\overline g~} X$. Hence $\overline g$ is a monomorphism.
\end{proof}

\begin{thm}\label{main1}
Assume that the abelian category  $\oB$ is hereditary. Let $\V$ be a funtorially finite subcategory which is closed under direct sums and summands. Then $\overline {\V}$ is support tilting if $\mathcal P\subseteq \V$ and $\overline {\V^{\bot_1}}\subseteq \Fac \overline \V$. 
\end{thm}

\begin{proof}
Since $\V$ is closed under direct sums and direct summands. It follows that
$\overline{\V}$ is closed under direct sums and direct summands.
Moreover, $\V$ is funtorially finite implies that $\overline{\V}$ is funtorially finite.
\smallskip

Since $\oB$ is hereditary, the condition $\Ext^2_{\oB}(\overline{\V},-)=0$ is satisfied.
\smallskip

Since $\overline {\V^{\bot_1}}\subseteq \Fac \overline \V$, by Proposition \ref{imp}, $\Ext^1_{\oB}(\overline{\V},\overline{\V})=0$. 
\smallskip

Let $0\neq Y\in \oB$ such that $\Ext^1_{\oB}(\overline \V, Y)=0$ and $Y$ is a subquotient of an object $ V\in \overline {\V}$. Then we have an epimorphism and a monomorphism $\xymatrix{V \ar@{->>}[r]^{\overline v} &T & Y \ar@{ >->}[l]_{\overline y}}$. Since $\overline {\V}$ is contravariantly finite, we can assume that $\overline v$ is a right $\overline {\V}$-approximation of $T$. Then we have the following commutative diagram of short exact sequences in $\oB$.
$$\xymatrix{
&&0 \ar[d] &0 \ar[d]\\
0 \ar[r] & W \ar[r] \ar@{=}[d] & U \ar[r] \ar[d] & Y \ar[r] \ar[d]^{\overline y} &0\\
0 \ar[r] & W \ar[r] & V \ar[r]^{\overline v} \ar[d]^{\overline z_1} & T \ar[r] \ar[d]^{\overline z_2} &0\\
&& Z \ar@{=}[r] \ar[d] & Z \ar[d]\\
&&0 &0
}
$$
Now we show that $\overline z_1$ is a right $\overline {\V}$-approximation.
Let $v':V'\to Z$ be any morphism where $V'\in \V$. Since $\Ext^1_{\oB}(\overline \V, Y)=0$, there is a morphism $\overline t:  V'\to  T$ such that $\overline v'= \overline {z_2t}$. Since $\overline v$ is a right $\overline {\V}$-approximation of $T$, there is a morphism $\overline v''\colon V'\to V$ such that $\overline t=\overline {vv''}$. Hence $\overline v'=\overline {z_2vv''}=\overline {z_1v''}$.


There is a right $\V$-approximation $z_1':V'\to Z$, then $\overline z_1'$ is also a right $\overline \V$-approximation. Moreover, it is an epimorphism. On the other hand, $z_1'$ admits the following commutative diagram
$$\xymatrix{
\Omega Z \ar[r] \ar@{=}[d] &W' \ar[r] \ar[d] &V' \ar[d]^{z_1'} \ar@{-->}[r] &\\
\Omega Z \ar[r]  &P_Z \ar[r]^p  &Z \ar@{-->}[r] &
}
$$
where $P_Z\in \mathcal P$. Since $\mathcal P\subseteq \V$, in the $\EE$-triangle $\xymatrix{W'\ar[r] &V'\oplus P_Z \ar[rr]^-{\alpha=\svech{z_1'}{p}} &&Z \ar@{-->}[r] &}$, $\alpha$ is still a right $\V$-approximation of $Z$. Then $W'\in \V^{\bot_1}$.

If $W'\in \K$, then $\overline z_1'$ becomes an isomorphism, thus $Z\in \overline \V$. Hence $T=Y\oplus Z$ in $\oB$ since $\Ext^1_{\oB}(\overline \V,Y)=0$. Then $Y$ is a quotient of $V\in \overline \V$.

If $W'\notin \K$, let $W'=W''\oplus K$ where $W''\in \h_\C$ and $K\in \K$. By Corollary \ref{ses}, we have a short exact sequence $0\to  W'' \to V'\xrightarrow{\overline z_1'} Z\to 0$. Since $U$ is a direct summand of $W''\oplus V$ in $\oB$, we have $U\in \overline {\V^{\bot_1}}\subseteq \Fac \overline \V$. Hence $Y$ is a quotient of an object in $\overline \V$.
\end{proof}

By Theorem \ref{main1}, we have the following corollary, which generalizes \cite[Theorem 2.2]{HJ}.

\begin{cor}\label{cor1}
Let $\oB$ be hereditary and $\M$ be a cluster tilting subcategory. Then $\overline {\M}$ is a support tilting subcategory in $\oB$.
\end{cor}




If we assume that $\B$ is a 2-Calabi-Yau triangulated category, by the result in \cite{LZ}, any fully rigid subcategory is a cluster tilting subcategory. Hence $\h=\B$ and $\oB=\B/\C$. Moreover, when we assume that $\oB$ is hereditary and has finite length, since $\oB$ is also Krull-Schmidt and Hom-finte, we get any subcategory of $\oB$ is funtorially finite. Then we have the following observation.

\begin{prop}
Let $\B$ be a 2-Calabi-Yau triangulated category with shift functor $[1]$ and $\C$ be a cluster tilting subcategory. When $\oB$ is hereditary and has finite length, there is a one-to-one correspondence between the weak cluster tilting subcategories $\X$ and support tilting subcategories $\overline \W\subseteq \oB$.
\end{prop}

\begin{proof}
By \cite[Theorem 3.5]{HJ}, any support tilting subcategory $\overline \W$ can be lifted to a weak cluster tilting subcategory $\X$ such that $\overline \X=\overline \W$.

Now let $\X$ be a weak cluster tilting subcategory. Then $\X=\add(\W\cup\T)$ where $\W=\X_\C$ and $\T=\{C\in \C \text{ }|\text{ }\EE(C,\W)=0 \}$. A weak cluster tilting subcategory is obviously closed under direct sums and direct summands. As in the proof of Theorem \ref{main1}, we also have $\Ext^1_{\oB}(\overline \X,\overline \X)=0$ and $\Ext^2_{\oB}(\overline \X,\overline \X)=0$. $\overline \X$ is funtorially finite under our assumptions. Now let $0\neq Y\in \oB$ such that $\Ext^1_{\oB}(\overline \X, Y)=0$ and $Y$ is a subquotient of an object $X\in \overline \X$. Just as in the proof of Theorem \ref{main1}, we can get a commutative diagram of short exact sequences in $\oB$
$$\xymatrix{
&&0 \ar[d] &0 \ar[d]\\
0 \ar[r] & V \ar[r] \ar@{=}[d] & U \ar[r] \ar[d] & Y \ar[r] \ar[d]^{\overline y} &0\\
0 \ar[r] & V \ar[r] & X \ar[r]^{\overline x} \ar[d]^{\overline z_1} & T \ar[r] \ar[d]^{\overline z_2} &0\\
&& Z \ar@{=}[r] \ar[d] &Z \ar[d]\\
&&0 &0
}
$$
where $\overline z_1$ is a right $\overline \X$-approximation and $U\in \overline \X^{\bot_1}=\W^{\bot_1}\cap \B_\C$ (this equation is followed by \cite[Lemma 3.3]{HJ}).
Morphism $z_1$ admits a triangle $(\bigstar) \text{ } Z[-1]\xrightarrow{u} U'\to X\xrightarrow{z_1} Z$. By Corollary \ref{ses} we have a short exact sequence $0\to  U'\to  X\xrightarrow{\overline z_1}  Z\to 0$. Hence $U\simeq U'$ in $\oB$ and $u$ factors through $\C$. If we can show that $U'\in \T^{\bot_1}$, then $U\in \X^{\bot_1}=\X$, hence $ Y$ is a quotient of an object in $\overline \X$.

Let $C\in \T$ and $c:C\to Z$ be any morphism. Then $c$ factors through $z_1$ since $u[1]c=0$. Thus we have an exact sequence when applying $\Hom_{\B}(\T,-)$ to triangle $(\bigstar)$.
$$\Hom_{\B}(\T,\B)\xrightarrow{\Hom_{\B}(\T,-u[1])=0}\Hom_{\B}(\T,U'[1])\to \Hom_{\B}(\T,X[1])=0.$$
Hence $\Hom_{\B}(\T,U'[1])=0$ and $U'\in \T^{\bot_1}$.
\end{proof}
\medskip

\section{Main result II}

In this section, we assume $\B$ is $\EE$-finite and has an $ARS$-duality $(\tau,\eta)$ defined in \cite[Definition 3.4]{INP}, $\oB$ is hereditary and has finite length. We also assume $\B$ satisfies condition (WIC) (\cite[Condition 5.8]{NP}):

\begin{itemize}
\item If we have a deflation $h: A\xrightarrow{~f~} B\xrightarrow{~g~} C$, then $g$ is also a deflation.
\item If we have an inflation $h: A\xrightarrow{~f~} B\xrightarrow{~g~} C$, then $f$ is also an inflation.
\end{itemize}

Note that this condition automatically holds on triangulated categories and Krull-Schmidt exact categories. This condition is needed in the proofs of some results in \cite{LZII} which we will refer to.

\begin{defn}
Let $\W\subseteq \h_\C$. We call a subcategory $\X\subseteq \h$ a lifting (with respect to $\h$) of $\overline {\W}$ if $\overline {\W}=\overline {\X}$.
\end{defn}

\begin{lem}\label{lift0}
If $\W\subseteq \h_\C$ has a lifting $\X$ such that $\X$ is maximal $\C$-rigid, then $\X=\add (\W\cup \T)$, where $\T=\{C\in \C \text { }| \text{ } \EE(C,\W)=0\}$.
\end{lem}

\begin{proof}
Let $X'\in \X_\W$ be an indecomposable object. Then $X'\in \C$. By \cite[Lemma 3.6]{LZII}, $\EE(X',\X)=0$, hence $X'\in \T$ and $\X\subseteq \add (\W\cup \T)$.

Since $\EE(\T,\X)=0$, by \cite[Lemma 3.6]{LZII}, $\add (\X\cup \T)$ is also $\C$-rigid. But $\X$ is maximal $\C$-rigid, we have $\T\subseteq \X$. Hence $\X=\add (\W\cup \T)$.
\end{proof}

\begin{thm}\label{main2.0}
Let $\oB$ be hereditary and have finite length. Let $\W\subseteq \h_\C$ such that $\overline \W$ is a support tilting subcategory of $\oB$. Then $\overline \W$ has a lifting $\X$ such that $\X$ is maximal $\C$-rigid.
\end{thm}

To show this theorem, the following lemmas are needed.

\begin{lem}\label{proj}
Let $\xymatrix{X\ar[r]^{f} &Y \ar[r]^{g} &Z \ar@{-->}[r] &}$ be an almost split extension in $\B$. Then $X\in \D$ if and only if $Z$ is projective in $\oB$.
\end{lem}

\begin{proof}
If $X\in \D$, then $Z\in \h_\C$, otherwise the $\EE$-triangle splits. If $Z$ is non-projective in $\oB$, we get a non-split $\EE$-triangle $\xymatrix{\Omega C^1\ar[r]^b &\Omega C^2\ar[r] &Z \ar@{-->}[r] &}$ where $\Omega C^i\in \Omega \C$ and $\Hom_\B(b,K)$ is an epimorphism if $K\in \K$. Then we have the following commutative diagram
$$\xymatrix{
\Omega C^1\ar[r]^b \ar[d]_c &\Omega C^2\ar[r] \ar[d] &Z \ar@{=}[d] \ar@{-->}[r] &\\
X \ar[r]_f &Y \ar[r]_g &Z \ar@{-->}[r] &.
}
$$
Since $X\in \K$, $c$ factors through $b$, hence $1_Z$ factors through $g$, a contradiction.\\
Now assume we have a non-split $\EE$-triangle $\xymatrix{X\ar[r]^x &B \ar[r] &K \ar@{-->}[r] &}$ where $K\in \K$, then we get the following commutative diagram
$$\xymatrix{
X \ar[r]^f \ar@{=}[d] &Y \ar[d] \ar[r]^g &Z  \ar[d]^z \ar@{-->}[r] &\\
X\ar[r]^x &B \ar[r]^k &K \ar@{-->}[r] &.
}
$$
Since $Z\in \Omega \C$, $z$ factors through $\mathcal P$, then $z$ factors through $k$, which implies $1_X$ factors through $f$, hence $f$ is a section, a contradiction. Hence $X\in \K^{\bot_1}=\D$.
\end{proof}

\begin{lem}\label{tau1}
$\tau\Omega \C_{\mathcal P}=\D_{\mathcal I}$.
\end{lem}

\begin{proof}
Let $X\in \tau\Omega \C_{\mathcal P}$. Then $X=\tau\Omega C_X$ where $C_X\in \C_{\mathcal P}$. Since $\mathbb{D} \EE(\K,X)\simeq \Hom_{\uB}(\Omega C_X, \K)=0$, we get $X\in \D_{\mathcal I}$. Hence $\tau\Omega \C_{\mathcal P}\subseteq \D_{\mathcal I}$.

Let $Y\in \D_{\mathcal I}$ be indecomposable, by Lemma \ref{proj}, we have the following two $\EE$-triangles
$$\xymatrix{Y \ar[r] &Z \ar[r] &X \ar@{-->}[r] &} \text{ and } \xymatrix{X \ar[r] &P \ar[r] &C \ar@{-->}[r]&}$$
where $X\in \Omega \C_{\mathcal P}$, $P\in \mathcal P$ and $C\in \C_{\mathcal P}$. This implies $\tau\Omega \C_{\mathcal P}\supseteq \D_{\mathcal I}$.
Hence $\tau\Omega \C_{\mathcal P}=\D_{\mathcal I}$.
\end{proof}

{\bf Now we give the proof of Theorem \ref{main2.0}.}

\begin{proof}
By Lemma \ref{lift0}, we only need to show that $\X'=\add (\W\cup \T)$ is maximal $\C$-rigid, where $\T=\{C\in \C \text { }| \text{ } \EE(C,\W)=0\}$.

Since $\Ext^1_{\oB}(\overline \W,\overline \W)$, by Proposition \ref{imp}, $\W$ is $\C$-rigid. Since $\EE(\T,\W)=0$, we get $\X'$ is also $\C$-rigid. We show it is maximal $\C$-rigid.

For a subcategory $\C_1\subseteq \C$, if $\add(\X'\cup \C_1)$ is $\C$-rigid, then we have $\EE(\C_1,\X')=0$. Then $\C_1\subseteq \T$.

Let $\X''\subseteq \h_\C$ such that $\add(\X'\cup \X'')$ is $\C$-rigid. By Proposition \ref{imp} we have $\Ext^1_{\oB}(\overline \W,\overline \X'')=0$. For an object $X\in \X''$, let $\Sigma D_X$ be an injective object of $\oB$ such that $\Hom_{\oB}(X,\Sigma D_X)\neq 0$. We can get $0\neq \Hom_{\B/\mathcal I}(X,\Sigma D_X)\simeq \EE(X,D_X)$. By Proposition \ref{tau1}, there is an object $C_X\in \C$ such that $\tau \Omega C_X=D_X$. Hence $0\neq \EE(X,D_X)=\EE(X,\tau\Omega C_X)\simeq  \mathbb{D} \Hom_{\uB}(\Omega C_X, X)\simeq  \mathbb{D} \EE(C_X, X)$. Then $C_X\notin \T$. Hence there is an object $W'\in \W$ such that $\EE(C_X,W')\neq 0$ and we have shown:
$$\Hom_{\oB}(X,\Sigma D_X)\neq 0\text{ }\Rightarrow \text{ there is an object } W'\in\W \text{ such that } \Hom_{\oB}(W',\Sigma D_X)\neq 0.$$
It follows from \cite[Lemma 3.4]{HJ} that $X$ is a subquotient of an object from $\overline \W$. But $\Ext^1_{\oB}(\overline \W, X)=0$, then $X$ is a quotient of an object of $\overline \W$ since $\overline \W$ is support tilting.
We have a short exact sequence $0\to Y\to W_X\to X\to 0$ in $\oB$. The long exact Ext sequence implies that $\Ext^1_{\oB}(\overline \W,Y)=0$. Since $Y$ is a subquotient of $ W_X$, then $Y$ is also a quotient of an object from $\overline \W$. Hence we have a short exact sequence $0\to Y'\to W_Y\to  Y\to 0$ in $\oB$. We have $\Ext^1_{\oB}(X, W_Y)=0$. Since $\oB$ is hereditary, we have $\Ext^2_{\oB}(X,Y')=0$, hence $\Ext^1_{\oB}(X,Y)=0$. Now the first short exact sequence splits, which implies $X$ is a direct summand of $W_X$. Since $\overline \W$ is closed under direct summands and $X$ is indecomposable, we have $X\in \W$.
Hence $\X$ is maximal $\C$-rigid.\\
\end{proof}




Let $B$ be an object in $\B$. We designate $\overline B^{\bot_1}$ (resp. $\overline B^{\bot}$) by the subcategory $\{B'\in B \text{ }|\text{ }\Ext^1_{\oB}(B,B')=0 \}$ (resp. $\{B'\in B \text{ }|\text{ }\Hom_{\oB}(B,B')=0 \}$) in $\oB$.

\begin{prop}\label{ori}
Assume $\C=\add C$ such that $C$ is basic. There is a one-to-one correspondence between the following classes of subcategories:
\begin{itemize}
\item[(a)] Maximal $\C$-rigid subcategories $\X=\add X\subseteq \h$.
\item[(b)] Support tilting subcategories $\overline \W$ of $\oB$ such that $\W=\add W\subseteq \h_\C$.
\end{itemize}
\end{prop}

\begin{proof}
By Theorem \ref{main2.0}, any support tilting object $W\in\oB$ has a lifting $\X=\add X$ which is maximal $\C$-rigid. Now let $X=W\oplus C_X$, where $W\in \h_\C$ and $C_X\in \C$, $\X=\add X$ is maximal $\C$-rigid. We show that $W$ is a support tilting object in $\oB$. By our assumption, $\X$ is a funtorially finite subcategory and $\mathcal P\subseteq \X$. By Theorem \ref{main1}, it is enough to show that $\overline {\X^{\bot_1}}\subseteq \Fac \overline \X$.\\
By \cite[Theorem 3.11]{LZII}, $(W,\Omega C_X)$ is a support $\tau_{\oB}$-tilting pair in $\oB$. By \cite[Corollary 2.13]{AIR}, we have $\overline W^{\bot_1}\cap \overline {\Omega C_X}^{\bot}=\Fac \overline \X$. But $\overline {\Omega C_X}^{\bot}=\overline {C_X^{\bot_1}}$. Hence we have
$$\overline {\X^{\bot_1}}=\overline {W^{\bot_1}}\cap \overline {C_X^{\bot_1}}\subseteq \overline W^{\bot_1}\cap \overline {\Omega C_X}^{\bot}=\Fac \overline \X.$$
\end{proof}

By Proposition \ref{ori} and \cite[Theorem 3.11]{LZII}, we get the following corollary.

\begin{cor}
Assume $\C=\add C$ such that $C$ is basic. An object $W\in \oB$ is support tilting if and only if it admits a support $\tau_{\oB}$-tilting pair $(W, \Omega C_X)$ in $\oB$, where $\add(W\oplus \C_X)$ is the lifting of $W$ given by Theorem \ref{main2.0}.
\end{cor}







\section{Example}
In this section, we give an example to explain our main results.

\begin{exm}\label{ex1}
Let $\Lambda$ be the $k$-algebra given by the quiver
$$
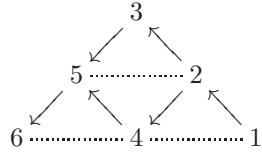
\xymatrix@C=0.4cm@R0.4cm{
&&3 \ar[dl]\\
&5 \ar[dl] \ar@{.}[rr] &&2 \ar[dl] \ar[ul]\\
6 \ar@{.}[rr] &&4 \ar[ul] \ar@{.}[rr] &&1 \ar[ul]}$$
with mesh relations. The Auslander-Reiten quiver of $\B:=\mod\Lambda$ is given by
$$
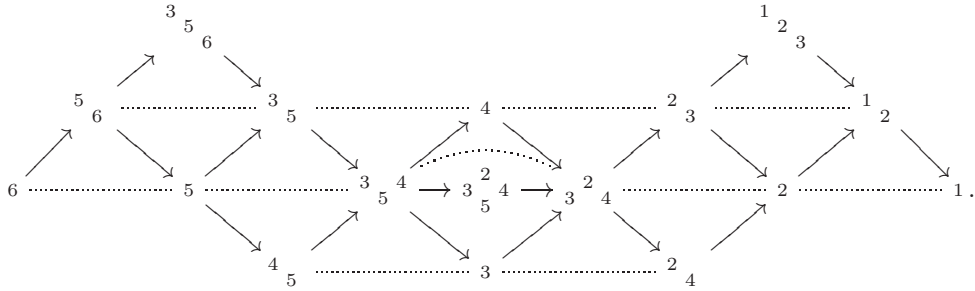
\xymatrix@C=0.4cm@R0.4cm{
&&{\begin{smallmatrix}
3&&\\
&5&\\
&&6
\end{smallmatrix}} \ar[dr] &&&&&&{\begin{smallmatrix}
1&&\\
&2&\\
&&3
\end{smallmatrix}} \ar[dr]\\
&{\begin{smallmatrix}
5&&\\
&6&
\end{smallmatrix}} \ar[ur] \ar@{.}[rr] \ar[dr] &&{\begin{smallmatrix}
3&&\\
&5&
\end{smallmatrix}} \ar@{.}[rr] \ar[dr] &&{\begin{smallmatrix}
4
\end{smallmatrix}} \ar@{.}[rr] \ar[dr] &&{\begin{smallmatrix}
2&&\\
&3&
\end{smallmatrix}} \ar[ur] \ar@{.}[rr] \ar[dr] &&{\begin{smallmatrix}
1&&\\
&2&
\end{smallmatrix}} \ar[dr]\\
{\begin{smallmatrix}
6
\end{smallmatrix}} \ar[ur] \ar@{.}[rr] &&{\begin{smallmatrix}
5
\end{smallmatrix}} \ar[ur] \ar@{.}[rr] \ar[dr] &&{\begin{smallmatrix}
3&&4\\
&5&
\end{smallmatrix}} \ar[ur] \ar[r] \ar[dr] \ar@{.}@/^15pt/[rr] &{\begin{smallmatrix}
&2&\\
3&&4\\
&5&
\end{smallmatrix}} \ar[r] &{\begin{smallmatrix}
&2&\\
3&&4
\end{smallmatrix}} \ar[ur] \ar@{.}[rr] \ar[dr] &&{\begin{smallmatrix}
2
\end{smallmatrix}} \ar[ur] \ar@{.}[rr] &&{\begin{smallmatrix}
1
\end{smallmatrix}}.\\
&&&{\begin{smallmatrix}
4&&\\
&5&
\end{smallmatrix}} \ar[ur] \ar@{.}[rr] &&{\begin{smallmatrix}
3
\end{smallmatrix}} \ar[ur] \ar@{.}[rr] &&{\begin{smallmatrix}
2&&\\
&4&
\end{smallmatrix}} \ar[ur]
}$$
We denote by ``~$\bullet$" in the Auslander-Reiten quiver the indecomposable objects belong to a subcategory and by ``~$\circ$'' the indecomposable objects do not belong to it.
$$
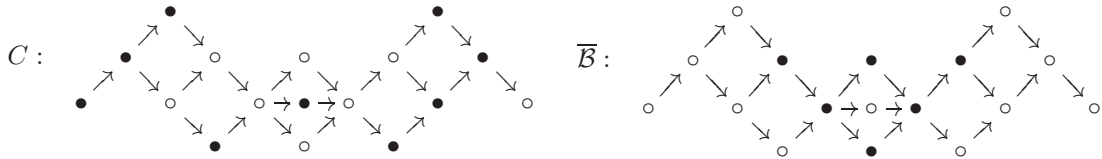
\xymatrix@C=0.2cm@R0.2cm{
&&&\bullet \ar[dr] &&&&&&\bullet \ar[dr]\\
{C:} &&\bullet \ar[ur]  \ar[dr] &&\circ  \ar[dr] &&\circ  \ar[dr] &&\circ  \ar[ur]  \ar[dr] &&\bullet \ar[dr]\\
&\bullet \ar[ur]  &&\circ \ar[ur]  \ar[dr] &&\circ \ar[ur] \ar[r] \ar[dr] &\bullet \ar[r] &\circ \ar[ur] \ar[dr] &&\bullet \ar[ur] &&\circ\\
&&&&\bullet \ar[ur] &&\circ \ar[ur] &&\bullet \ar[ur]
\\} \quad
\xymatrix@C=0.2cm@R0.2cm{
&&&\circ \ar[dr] &&&&&&\circ \ar[dr]\\
{\oB:} &&\circ \ar[ur]  \ar[dr] &&\bullet  \ar[dr] &&\bullet  \ar[dr] &&\bullet  \ar[ur]  \ar[dr] &&\circ \ar[dr]\\
&\circ \ar[ur]  &&\circ \ar[ur]  \ar[dr] &&\bullet \ar[ur] \ar[r] \ar[dr] &\circ \ar[r] &\bullet \ar[ur] \ar[dr] &&\circ \ar[ur] &&\circ\\
&&&&\circ \ar[ur] &&\bullet \ar[ur] &&\circ \ar[ur]
}$$
where $\C=\add C$ is fully rigid. $\oB$ is hereditary and has finite length.
We have the following maximal $\C$-rigid objects of $\B$ which satisfy the assumption in Theorem \ref{main1}:
\vspace{1mm}
$${\begin{smallmatrix}
\ &3&\
\end{smallmatrix}} \oplus
\begin{smallmatrix}
3&&4\ \\
&5&
\end{smallmatrix} \oplus {\begin{smallmatrix}
\ &4&\
\end{smallmatrix}}\oplus \Lambda, \quad {\begin{smallmatrix}
\ &3&\
\end{smallmatrix}} \oplus
\begin{smallmatrix}
&2&\ \\
3&&4
\end{smallmatrix} \oplus {\begin{smallmatrix}
\ &4&\
\end{smallmatrix}}\oplus \Lambda,\quad{{\begin{smallmatrix}
\ &3&\
\end{smallmatrix}} \oplus
\begin{smallmatrix}
&2&\ \\
3&&4
\end{smallmatrix}} \oplus {\begin{smallmatrix}
2&&\\
&3&
\end{smallmatrix}} \oplus \Lambda, \quad {{\begin{smallmatrix}
\ &4&\
\end{smallmatrix}} \oplus
\begin{smallmatrix}
&2&\ \\
3&&4
\end{smallmatrix}}\oplus \Lambda \oplus {\begin{smallmatrix}
2&&\\
&4&
\end{smallmatrix}}.
$$
\vspace{1mm}
By Theorem \ref{main2.0}, they are the liftings of the following support tilting objects in $\oB$:
\vspace{1mm}$${\begin{smallmatrix}
\ &3&\
\end{smallmatrix}} \oplus
\begin{smallmatrix}
3&&4\ \\
&5&
\end{smallmatrix} \oplus {\begin{smallmatrix}
\ &4&\
\end{smallmatrix}},\quad{\begin{smallmatrix}
\ &3&\
\end{smallmatrix}} \oplus
\begin{smallmatrix}
&2&\ \\
3&&4
\end{smallmatrix} \oplus {\begin{smallmatrix}
\ &4&\
\end{smallmatrix}},\quad {{\begin{smallmatrix}
\ &3&\
\end{smallmatrix}} \oplus
\begin{smallmatrix}
&2&\ \\
3&&4
\end{smallmatrix}} \oplus {\begin{smallmatrix}
2&&\\
&3&
\end{smallmatrix}}, \quad {{\begin{smallmatrix}
\ &4&\
\end{smallmatrix}} \oplus
\begin{smallmatrix}
&2&\ \\
3&&4
\end{smallmatrix}}.
$$\\
Moreover,
${\begin{smallmatrix}
\ &3&\
\end{smallmatrix}} \oplus
\begin{smallmatrix}
3&&4\ \\
&5&
\end{smallmatrix} \oplus {\begin{smallmatrix}
\ &4&\
\end{smallmatrix}},\quad {\begin{smallmatrix}
\ &3&\
\end{smallmatrix}} \oplus
\begin{smallmatrix}
&2&\ \\
3&&4
\end{smallmatrix} \oplus {\begin{smallmatrix}
\ &4&\
\end{smallmatrix}},\quad {{\begin{smallmatrix}
\ &3&\
\end{smallmatrix}} \oplus
\begin{smallmatrix}
&2&\ \\
3&&4
\end{smallmatrix}} \oplus {\begin{smallmatrix}
2&&\\
&3&
\end{smallmatrix}}$
are tilting objects, and
$({{\begin{smallmatrix}
\ &4&\
\end{smallmatrix}} \oplus
\begin{smallmatrix}
&2&\ \\
3&&4
\end{smallmatrix}}, {\begin{smallmatrix}
3&&\\
&5&
\end{smallmatrix}})
$
is a support $\tau_{\oB}$-tilting pair in $\oB$.\\
We also have a cluster tilting object:
$$
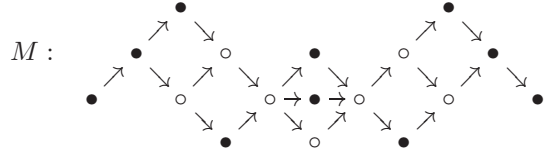
\xymatrix@C=0.2cm@R0.2cm{
&&&\bullet \ar[dr] &&&&&&\bullet \ar[dr]\\
{M:} &&\bullet \ar[ur]  \ar[dr] &&\circ  \ar[dr] &&\bullet  \ar[dr] &&\circ  \ar[ur]  \ar[dr] &&\bullet \ar[dr]\\
&\bullet \ar[ur]  &&\circ \ar[ur]  \ar[dr] &&\circ \ar[ur] \ar[r] \ar[dr] &\bullet \ar[r] &\circ \ar[ur] \ar[dr] &&\circ \ar[ur] &&\bullet\\
&&&&\bullet \ar[ur] &&\circ \ar[ur] &&\bullet \ar[ur]
\\}
$$
by Corollary \ref{cor1}, we get a support tilting subcategory $\add ({\begin{smallmatrix}
4
\end{smallmatrix}})$ of $\oB$.
\\In fact, according to Proposition \ref{ori}, any maximal $\C$-rigid object admits a support tilting object in $\oB$.
\end{exm}

\end{document}